\documentclass[leqno,10pt,a4paper,intlimits]{amsart}

\usepackage{amsmath,amssymb,amsthm}
\usepackage{enumerate,float}
\usepackage{graphics,color}
\def\PSforPDF#1{#1}

\def\calT{\mathcal{T}}
\def\calS{\mathcal{S}}
\def\calE{\mathcal{E}}

\def\LS{T} 

\def\C{\mathrm{C}}
\def\W{\mathrm{W}}
\def\Ell{\mathrm{L}}

\def\dom{D}
\def\dd{\mathrm{d}}
\def\ee{\mathrm{e}}

\def\NN{\mathbb{N}}
\def\RR{\mathbb{R}}
\def\dt{\tfrac{\dd}{\dd t}}
\def\Id{I}

\def\vect#1#2{\binom{#1}{#2}}
\def\tvect#1#2{\tbinom{#1}{#2}}

\newtheorem{theorem}{Theorem}[section]

\theoremstyle{definition}

\newtheorem{assumption}[theorem]{Assumption}
\newtheorem{example}[theorem]{Example}

\begin{document}

\title[Splitting for delay equation]{Operator splitting for nonautonomous delay equations}

\author[A. B\'{a}tkai]{Andr\'{a}s B\'{a}tkai}
\address{E\"otv\"os Lor\'and University, Institute of Mathematics, 1117 Budapest, P\'{a}zm\'{a}ny P. s\'{e}t\'{a}ny 1/C, Hungary.}
\email{batka@cs.elte.hu}

\author[P. Csom\'{o}s]{Petra Csom\'{o}s}
\address{Leopold--Franzens--Universit\"{a}t Innsbruck, Institut f\"{u}r Mathematik, Technikerstra{\ss}e 13, 6020 Innsbruck, Austria.}
\email{petra.csomos@uibk.ac.at}

\author[B. Farkas]{B\'{a}lint Farkas}
\address{E\"otv\"os Lor\'and University, Institute of Mathematics, 1117 Budapest, P\'{a}zm\'{a}ny P. s\'{e}t\'{a}ny 1/C, Hungary.}
\email{fbalint@cs.elte.hu}


\date\today

\begin{abstract}
We provide a general product formula for the solution of nonautonomous abstract delay equations. After having shown the convergence we obtain estimates on the order of convergence for differentiable history functions. Finally, the theoretical results are demonstrated by some typical numerical examples.
\end{abstract}
\allowdisplaybreaks
\maketitle


\section{Introduction}

Operator splitting is a widely used time discretization method for the numerical solution of complicated equations. The importance and main applications of these procedures is described, for example, in the monographs by Farag\'{o} and Havasi \cite{Farago-Havasi_book}, Holden et al.~\cite{Holden-Karlsen-Lie-Risebro} and Lubich \cite{Lubich08}.

\medskip\noindent The present paper investigates a special operator splitting for a class of nonautonomous delay differential equations. This method, which can be applied to equations with distributed delays very effectively, was first investigated in Csom\'{o}s and Nickel \cite{Csomos-Nickel} and in B\'{a}tkai, Csom\'{o}s and Nickel \cite{Batkai-Csomos-Nickel} in the autonomous case. Recall from the Introduction in Bellen and Zennaro \cite{Bellen-Zennaro} that delay equations with distributed delay, especially those where the delay term is not separated from zero, are particularly difficult to solve numerically. Nevertheless, as we shall see, splitting methods work quite well even in the latter case.

\medskip\noindent To motivate this approach, let us consider the following equation.
\begin{equation}
\left\{
\begin{aligned}
\dot u(t)&=b(t)u(t)+\displaystyle\int_{-1}^{0} \mu(t,\sigma)u(t+\sigma)\dd\sigma, \qquad t\ge s, \\
u(s)&=x\in\RR, \\
u(s+\sigma)&=f(\sigma),\qquad \sigma\in [-1,0],
\end{aligned}
\right.
\nonumber
\end{equation}
where $b\in \C^1_b(\RR)$, $\mu\in\Ell^\infty(\RR\times [-1,0])$, and $t\mapsto \mu(t,\sigma)\in \C^1_b(\RR)$ for all $\sigma\in[-1,0]$. In this case the delay operator $\Phi(t)$ (cf. equation \eqref{delay} below) is defined by
\begin{equation}
\Phi(t) g:=\int_{-1}^{0} \mu(t,\sigma)g(\sigma)\dd\sigma
\nonumber
\end{equation}
for all $g\in\Ell^1([-1,0])$.

\noindent Choosing a time step $h\in(0,1]$, first we start with  $x_0:=x$ and $f_0:=f$. Then we set
\begin{equation*}
x_1:= \ee^{hb(s)}(x_0+h\Phi(s)f_0)
\end{equation*}
and
\begin{equation*}
f_1(\sigma):= \begin{cases} \ee^{(h+\sigma)b(s)}(x_0+h\Phi(s)f_0), \quad \sigma\in [-h,0], \\ f_0(h+\sigma),\quad \sigma\in [-1,-h).\end{cases}
\end{equation*}
In the next step we repeat this procedure and replace $x_0$ with $x_1$, $f_0$ with $f_1$, and $s$ with $s+h$. Hence, we obtain an iteration process where in the $k^{\text{th}}$ step we have
\begin{equation}\label{eq:splitting_procedure_a}
\left\{
\begin{aligned}
x_{k}&:= \ee^{hb(s+h(k-1))}(x_{k-1}+h\Phi(s+h(k-1))f_{k-1}),\\
f_k(\sigma)&:= \begin{cases} \ee^{(h+\sigma)b(s+h(k-1))}(x_{k-1}+h\Phi(s+h(k-1))f_{k-1}), \, \sigma\in [-h,0], \\ f_{k-1}(h+\sigma),\quad \sigma\in [-1,-h).\end{cases}
\end{aligned}
\right.
\end{equation}

\noindent The aim of the present paper is to prove the convergence of this procedure also for more general equations. We do this by introducing an abstract setup allowing us a general convergence result of the procedure. Further, for differentiable initial function (i.e., classical solutions) we also obtain estimates on the order of convergence. In the following, we summarize some basic facts on product formulae for abstract evolution equations. Then in Section 2 we rewrite the nonautonomous delay equation as an abstract evolution equation and prove the convergence of a general product formula. Section 3 is devoted to the investigation of the order of convergence, and in Section 4 we present numerical examples demonstrating the power of this approach.


\medskip First, let us recall some general facts about splitting of nonautonomous equations.
Consider an evolution equation of the form
\begin{equation}\label{eq:perturbed}\tag{NCP}
\begin{cases}
\dt{u}(t)=(A(t)+B(t))u(t), \quad t\geq s\in\RR, \\
u(s)=y\in X.
\end{cases}
\end{equation}
We always suppose that this  equation \eqref{eq:perturbed} is well-posed, i.e., there is an evolution family $W$ (also called semi-dynamical system) solving it.  By an \emph{evolution family}  we mean a strongly continuous two-parameter family of bounded linear operators $W$ with the algebraic property $W(t,s)W(s,r)=W(t,r)$ and $W(t,t)=\Id$ for all $r\leq s\leq t$. For  well-posedness of nonautonomous evolution equations we refer to the surveys in Nagel and Nickel \cite{Nagel-Nickel}, Pazy \cite[Chapter 5]{Pazy} or Schnaubelt in \cite[Chapter VI.9]{Engel-Nagel}.

\medskip\noindent  Now, a splitting formula, as below,  is especially useful, if we are able to solve effectively the \emph{autonomous} Cauchy problems
\begin{align*}
\dt u(t)&=A(r)u(t) \\
\dt v(t)&=B(r)v(t)
\end{align*}
with appropriate initial conditions for every fixed $r$. This is usually the case, if the operators $A(r)$ and $B(r)$ are partial differential operators with time dependent coefficients, or time dependent multiplication operators, and this is particulary so for delay equations considered in this paper.

\medskip\noindent The following general convergence result can be proved, see B\'{a}tkai et al. \cite[Theorem 4.2]{Batkai-Csomos-Farkas-Nickel1}.
\begin{theorem}\label{th:firstproduct} Suppose the following:
\begin{enumerate}[a)]
\item The nonautonomous Cauchy problem corresponding to the operators $(A(\cdot)+ B(\cdot))$ is well-posed. We denote the evolution family solving \eqref{eq:perturbed} by $W$.
\item The operators $A(t)$ and $B(t)$ are generators of $C_0$-semigroups of type $(M,\omega)$ ($M\geq 1$ and $\omega\in \RR$), and
\begin{equation*}
(\omega, \infty) \subset \rho(A(t))\cap\rho(B(t)) \quad \mbox{for all } t \in \RR
\end{equation*}
and
\begin{equation*} \label{stab}
\sup_{s\in\RR}\Bigl\| \prod_{p=n}^{1} \bigl(\ee^{\frac{t}{n}A(s- p\frac{t}{n})} \ee^{\frac{t}{n}B(s- p\frac{t}{n})}\bigr) \Bigr\|
\leq M \ee^{\omega t}.
\end{equation*}
\item  The maps
\begin{equation*}
t\mapsto R(\lambda,A(t))y,\qquad t\mapsto R(\lambda,B(t))y
\end{equation*}
are continuous for all $\lambda>\omega$ and $y\in X$.
\end{enumerate}
Then one has the convergence
\begin{equation}
\label{eq:ev_product}
W(t,s)y= \lim_{n\to\infty} \prod_{i=0}^{n-1} \bigl(\ee^{\frac{t-s}{n}A(s+ i\frac{t-s}{n}} \ee^{\frac{t-s}{n}B(s+ i\frac{t-s}{n}}\bigr)y
\end{equation}
for all $y\in X$, locally uniformly in $s,t$ with $s\leq t$.
\end{theorem}


\section{Splitting for the delay equation}
Consider the \textbf{abstract delay equation} in the following form:
\begin{equation}
\left\{
\begin{aligned}
\tfrac{\dd }{\dd t}u(t)&=A(t)u(t)+\Phi(t) u_t, \qquad t\ge s, \\
u(s)&=x\in X,\, s\in\RR, \\
u_s&=f\in\Ell^1([-1,0];X)
\end{aligned}
\right.
\label{delay}
\end{equation}
on the Banach space $X$, where $A(t)$ generates a strongly continuous contraction semigroup on $X$ and $\Phi(t):\Ell^{1}([-1,0];X) \to X$ is a bounded and linear operator depending continuously on the parameter $t\in\RR$. The \textbf{history function} $u_t$ is defined by $u_t(\sigma):=u(t+\sigma)$ for $\sigma\in[-1,0]$.  Note that point delays are excluded from this context, but \emph{distributed delays}, even those that live up to $0$, are contained in this setting.

\medskip\noindent In order to rewrite \eqref{delay} as an abstract Cauchy problem, we take the product space $\mathcal{E}:=X\times \Ell^1([-1,0];X)$ equipped with $1$-sum norm, and the new unknown function as
\begin{equation}
t\mapsto\mathcal{U}(t):=\binom{u(t)}{u_t}\in\mathcal{E}.
\nonumber
\end{equation}
Then \eqref{delay} can be written as an abstract Cauchy problem on the space $\mathcal{E}$ in the following way:
\begin{equation}
\left\{
\begin{aligned}
\tfrac{\dd}{\dd t}\mathcal{U}(t)&=\mathcal{G}(t)\mathcal{U}(t), \qquad t\ge s, \\
\mathcal{U}(s)&=\tvect{x}{f}\in\mathcal{E},
\end{aligned}
\right.
\label{acp_delay}
\end{equation}
where the operator $\mathcal{G}(t)$ is given by the matrix
\begin{equation}
\mathcal{G}(t):=\begin{pmatrix} A(t) & \Phi(t) \\ 0 & \frac{\dd}{\dd\sigma} \end{pmatrix}
\end{equation}
on the domain
\begin{equation}
\dom(\mathcal{G}(t)):=\left\{\tvect{x}{f}\in \dom(A(t))\times \W^{1,1}([-1,0];X): \,f(0)=x  \right\}.
\nonumber
\end{equation}
\noindent As in B\'{a}tkai and Piazzera \cite[Corollary 3.5, Proposition 3.9]{Batkai-Piazzera} one can show that the delay equation \eqref{delay} and the abstract Cauchy problem \eqref{acp_delay} are equivalent, i.e., they have the same solutions. More precisely, the first coordinate of the solution of \eqref{acp_delay} always solves \eqref{delay}, i.e.,
 \begin{equation*}
 u(t)=\pi_1 \mathcal{U}(t),
 \end{equation*}
\noindent where $\pi_1$ is the projection to the first coordinate in $\mathcal{E}$.
Due to this equivalence, the delay equation is well-posed if and only if the operator $\mathcal{G}(t)$ generates an evolution family on the space $\mathcal{E}$.

\noindent Since the delay operators $\Phi(t)$ are bounded, the delay equation \eqref{delay} is well-posed, which follows form a much more general well-posedness result by Hadd, Rhandi and Schnaubelt \cite[Proposition 3.5]{Hadd-Rhandi-Schnaubelt}. That is, there is an \emph{evolution family} $\mathcal{W}$ such that for fixed $s$ and $t\geq s$ the function $u^{(s)}(t)=\pi_1\mathcal{W}(t,s)\tvect{x}{f}$ is a solution of \eqref{delay} for $\tvect{x}{f}\in \dom(\mathcal{G}(s))$. In particular, for $0\leq s\leq t$ we have
\begin{alignat*}{2}
&&\mathcal{W}(t,s)\tvect{x}{f}&=\vect{u^{(s)}(t)}{u^{(s)}_{t}},\\
&\mbox{where $u^{(s)}$ fulfills}&\qquad u^{(s)}(t)&=x+\int_s^{t} A(r)u^{(s)}(r)\dd r+\int_s^{t} \Phi(r) u^{(s)}_{r}\dd r\\
&\mbox{and}&\qquad u_t^{(s)}(r)&=u^{(s)}(t+r)\quad\mbox{for $r\in[-1,0]$, $t+r\geq s$}.
\end{alignat*}
Furthermore, we have the next relation
\begin{equation*}
u_t^{(s)}(r)=f(r+t-s) \quad \mbox{for $t+r<s$}.
\end{equation*}

\noindent Now we make the main assumptions implying the convergence of the splitting procedure. In the autonomous case, i.e., when $A(t)=A$ and $\Phi(t)=\Phi$, the following was  investigated in the papers by Csom\'{o}s and Nickel \cite{Csomos-Nickel} and B\'{a}tkai, Csom\'{o}s, and Nickel \cite{Batkai-Csomos-Nickel}.
\begin{assumption}\label{ass:2}
\begin{enumerate}[a)]
\item The operators $A(s)$ generate the strongly continuous contraction semigroups $(V^{(s)}(t))_{t\ge 0}$ on $X$ for all $s\in\RR$.
\item $\dom(A(s))=:D$ for all $s\in \RR$ and the function $s\mapsto R(1,A(s))y$ is continuous for all $y\in X$.
\item The delay operators $\Phi(s):\ \Ell^1([-1,0];X) \to X$ are bounded for all $s\in \RR$.
\item The function $s\mapsto \Phi(s)f$ is bounded and continuous for every $f\in \Ell^1([-1,0];X)$.
\end{enumerate}
\label{ass1}
\end{assumption}

\medskip\noindent  Let us now describe in detail the approximation procedure we will apply. We split the operator in \eqref{acp_delay} as
\begin{equation}
\mathcal{G}(t)=\mathcal{A}(t)+\mathcal{B}(t),
\nonumber
\end{equation}
where the sub-operators have the form
\begin{equation}\arraycolsep2pt
\begin{array}{rcll}\arraycolsep1em
\mathcal{A}(r)&:=&\begin{pmatrix} A(r) & 0 \\ 0 & \frac{\dd}{\dd\sigma} \end{pmatrix}, & \quad \dom(\mathcal{A}(r)):=\dom(\mathcal{G}(r)), \medskip \\
\mathcal{B}(r)&:=&\begin{pmatrix} 0 & \Phi(r) \\ 0 & 0 \end{pmatrix}, & \quad \dom(\mathcal{B}(r)):=\mathcal{E}.
\end{array}
\label{matrices-delay}
\end{equation}
Since $A(r)$ is a generator and $\Phi(r)$ is bounded, the operators $\mathcal{A}(r)$ and $\mathcal{B}(r)$ generate the strongly continuous semigroups $(\calT^{(r)}(t))_{t\ge 0}$ and $(\calS^{(r)}(t))_{t\ge 0}$, respectively. It is shown in B\'{a}tkai and Piazzera \cite[Theorem 3.25]{Batkai-Piazzera} that $\calT^{(r)}$ is given by
\begin{equation}
\calT^{(r)}(t):=\begin{pmatrix} V^{(r)}(t) & 0 \\ V^{(r)}_t & \LS(t) \end{pmatrix},
\nonumber
\end{equation}
where $(\LS(t))_{t\ge 0}$ is the nilpotent left shift semigroup defined by
\begin{equation}
(\LS(t)f)(\sigma):=
\begin{cases}
f(t+\sigma), & \quad\mbox{if}\quad \sigma\in[-1,-t), \\
0, & \quad\mbox{if}\quad \sigma\in[-t,0]
\end{cases}
\nonumber
\end{equation}
for all $f\in\Ell^1([-1,0];X)$, and $V^{(r)}_t$ is
\begin{equation}
(V^{(r)}_t x)(\sigma):=
\begin{cases}
V^{(r)}(t+\sigma)x, & \quad\mbox{if}\quad \sigma\in[-t,0], \\
0, & \quad\mbox{if}\quad \sigma\in[-1,-t)
\end{cases}
\nonumber
\end{equation}
for all $x\in X$. Since $\Phi(r)$ is a bounded operator, $\mathcal{B}(r)$ is also bounded on $\mathcal{E}$. Therefore, the semigroup $\calS^{(r)}$ generated by $\mathcal{B}(r)$ takes the form
\begin{equation}
\calS^{(r)}(t):=\ee^{t\mathcal{B}(r)}=\mathcal{I}+t\mathcal{B}(r)=\begin{pmatrix} I & t\Phi(r) \\ 0 & \widetilde I \end{pmatrix},
\nonumber
\end{equation}
where $I$, $\widetilde I$, and $\mathcal{I}$ denote the identity operators on $X$, $\Ell^1([-1,0];X)$, and $\mathcal{E}$, respectively. We then have the following general convergence result explaining the convergence of the procedure described in \eqref{eq:splitting_procedure_a}.

\begin{theorem}\label{thm:general}
Under  Assumption \ref{ass1} the solution of the abstract delay equation \eqref{delay} is given by the formula
\begin{equation*}
u^{(s)}(t)=\pi_1 \lim_{n\to\infty} \prod_{p=0}^{n-1} \begin{pmatrix} V^{(s+p\frac{t-s}{n})}(\tfrac{t-s}{n}) & 0 \\ V^{(s+p\frac{t-s}{n})}_{\frac{t-s}{n}} & \LS(\tfrac{t-s}{n}) \end{pmatrix} \begin{pmatrix} I & \tfrac{t-s}{n}\Phi(s+p\tfrac{t-s}{n}) \\ 0 & \widetilde I \end{pmatrix} \tvect{x}{f}.
\end{equation*}
\end{theorem}

\begin{proof}
The convergence is a direct consequence of Theorem \ref{th:firstproduct} applied to the generators $\mathcal{A}(r)$ and $\mathcal{B}(r)$. As mentioned above, by Hadd, Rhandi and Schnaubelt \cite{Hadd-Rhandi-Schnaubelt} the Cauchy problem associated to the operator $\mathcal{A}(t)+\mathcal{B}(t)$ is well-posed. Applying Theorem \ref{th:firstproduct}, we only have to check the stability assumption. Indeed, Assumption \ref{ass:2}.a) and b) and an argument using power series imply the continuity of $s\mapsto R(\lambda,A(s))$ for $\lambda$ in the connected component of $1$, hence condition c) on  Theorem \ref{th:firstproduct} is fulfilled. Now, to see stability one can use the same arguments as the ones appearing in the proof of Csom\'{o}s and  Nickel \cite[Theorem 4.2]{Csomos-Nickel}, we get that
\begin{align*}
\|\calT^{(r)}(t)\tvect{x}{f}\|_\calE &=\|V^{(r)}(t)x\|+\|V^{(r)}_tx+\LS(t)f\|_1\leq  \|x\|+\|f\|_1+t\|x\|\\
&\leq (1+t)(\|x\|+\|f\|_1)=(1+t)\|\tvect{x}{f}\|_1,
\end{align*}
hence
\begin{align*}
\|\calT^{(r)}(t)\tvect{x}{f}\|_\calE&\leq 1+t.
\intertext{One also obtains}
\|\calS^{(r)}(t)\| &\leq 1+t\|\Phi(r)\|.
\end{align*}
From this the stability can be obtained:
\begin{align*}
\Bigl\| \prod_{i=n}^{1} \begin{pmatrix} V^{(s-i\frac{t}{n})}(\frac{t}{n}) & 0 \\ V^{(s-i\frac{t}{n})}_{\frac{t}{n}} & \LS(\frac{t}{n}) \end{pmatrix} \begin{pmatrix} I & \frac{t}{n}\Phi(s-i\tfrac{t}{n}) \\ 0 & \widetilde I \end{pmatrix} \Bigr\| &\leq \Bigl(1+\frac{t}{n}\sup_{s\in\RR} \|\Phi(s)\|\Bigr)^n\Bigl(1+\frac{t}{n}\Bigr)^n \\
&\leq \ee^{\omega t},
\end{align*}
with $\omega = 1+\sup_{s\in\RR} \|\Phi(s)\|$.
\end{proof}


\section{Order of convergence}
We now investigate the order of convergence of the splitting method from the previous section. To this end, we have to make further assumptions on the operators involved, and of course, some regularity on the initial data need to be assumed, too.

\begin{assumption}\label{ass:convrate}~
\begin{enumerate}[a)]
\item The operator $A(s)$ is bounded and generates the  strongly continuous contraction semigroup $(V^{(s)}(t))_{t\ge 0}$ on $X$.
\item The delay operators $\Phi(s):\ \Ell^1([-1,0];X) \to X$ are bounded for all $s\in \RR$.
\item The function $s\mapsto A(s)x$ is bounded and locally Lipschitz continuous, i.e., for all ${T_0}>0$ there is $L_{T_0}\geq 0$ such that $$\|A(s)x-A(t)x\|\leq L_{T_0}\|x\||t-s|$$ for all $|t|,|s|\leq {T_0}$.
\item The function $s\mapsto \Phi(s)f$ is bounded and locally Lipschitz continuous, i.e., for all ${T_0}>0$ there is $L_{T_0}\geq 0$ such that $$\|\Phi(s)f-\Phi(t)f\|\leq L_{T_0}\|f\|_1|t-s|$$ for all $|t|,|s|\leq {T_0}$.
\end{enumerate}
\end{assumption}

\noindent These assumptions enable us to show the first order of convergence for classical solutions.
\begin{theorem}[{\bf Local error estimate}]\label{thm:approxt2}
Let $T_0>0$ be fixed. Then there is a constant $C>0$ such that
\begin{equation*}
\bigl\|\calT^{(s)}(h)\calS^{(s)}(h) \tvect{x}{f}-\mathcal{W}(s+h,s)\tvect{x}{f}\bigr\| \leq Ch^2\bigl(\|x\|+\|f\|_1+\|f'\|_1\bigr)
\end{equation*}
holds for all $h\in [0,1]$, $s\in [-T_0,T_0]$ and $f\in\W^{1,1}([-1,0];X)$, $x=f(0)$.
\end{theorem}
\begin{proof}
Recall from the previous section that for $0\leq s\leq s+h$ we have
\begin{alignat*}{2}
&&\mathcal{W}(s+h,s)\tvect{x}{f}&=\vect{u^{(s)}(s+h)}{u^{(s)}_{s+h}},\\
&\mbox{where $u^{(s)}$ fulfills}&\qquad u^{(s)}(s+h)&=x+\int_s^{s+h} A(r)u^{(s)}(r)\dd r+\int_s^{s+h} \Phi(r) u^{(s)}_{r}\dd r\\
&\mbox{and}&\qquad u_h^{(s)}(r)&=u^{(s)}(h+r)\quad\mbox{for $h+r\geq s$}.
\end{alignat*}
From this it follows that $u^{(s)}:[s,s+1]\to X$ is Lipschitz continuous with constant $L(\|x\|+\|f\|_1)$ with $L$ dependent only on $\|A\|_\infty$ and $\|\Phi\|_\infty$.
Let us calculate the product
\begin{equation}\label{eq:ize}
\calT^{(s)}(h)\calS^{(s)}\tvect{x}{f}(h)=\vect{V^{(s)}(h)x+hV^{(s)}(h)\Phi(s) f}{V^{(s)}_hx+hV^{(s)}_h\Phi(s) f+\LS(h)f},
\end{equation}
and compare the first component here with $u^{(s)}(s+h)$. We can write
\begin{align*}
&V^{(s)}(h)x+hV^{(s)}(h)\Phi(s) f-u^{(s)}(s+h)\\
&\quad=V^{(s)}(h)x+hV^{(s)}(h)\Phi(s) f-x-\int_s^{s+h} A(r)u^{(s)}(r)\dd r-\int_s^{s+h} \Phi(r) u^{(s)}_r\dd r,
\end{align*}
and by writing out the series expansion of $V^{(s)}(h)$ we obtain
\begin{align}
\notag&V^{(s)}(h)x+hV^{(s)}(h)\Phi(s) f-u^{(s)}(s+h)\\
\notag&\qquad=x+hA(s)x+h\Phi(s) f+O(h^2) -x-\int_s^{s+h} A(r)u^{(s)}(r)\dd r-\int_{s}^{s+h}\Phi(r) u^{(s)}_r\dd r\\
\label{eq:33}&\qquad=hA(s)x-\int_s^{s+h} A(r) u^{(s)}(r)\dd r+h\Phi(s) f-\int_s^{s+h} \Phi(r)u^{(s)}_r\dd r +O(h^2),
\end{align}
where $O(h^2)$ denotes a term bounded in norm by $C\cdot h^2(\|x\|+\|f\|_1)$ with a  constant $C$ that depends only on the bounds of $\|A\|_\infty$ and $\|\Phi\|_\infty$. We now can write
\begin{align*}
&\Bigl\|hA(s)x-\int_s^{s+h}A(r) u^{(s)}(r)\dd r\Bigr\|\leq \int_s^{s+h} \|A(s)x-A(r)u^{(s)}(r)\|\dd r\\
 &\qquad\leq  \int_s^{s+h} \|A(s)x-A(r)x\|\dd r+\int_s^{s+h} \|A(r)x-A(r)u^{(s)}(r)\|\dd r\\
 &\qquad \leq (L'\|x\|+\|A\|_\infty L(\|x\|+\|f\|_1)) \int_s^{s+h} (r-s)\dd r=O(h^2),
\end{align*}
where $L'$ is the Lipschitz constant of $A$ on $[-T_0,T_0+1]$. A very similar reasoning works for the other two terms in \eqref{eq:33}:
\begin{align*}
&\Bigl\|h\Phi(s) f-\int_s^{s+h} \Phi(r)u^{(s)}_r\dd r\Bigr\|\leq \int_s^{s+h} \|\Phi(s) f-\Phi(r) u^{(s)}_r\|\dd r\\
&\quad\leq\int_s^{s+h} \|\Phi(s) f-\Phi(r) f\|\dd r+ \int_s^{s+h} \|\Phi(r)f-\Phi(r)u^{(s)}_r\|\dd r\\
&\quad\leq L'\|f\|_1\int_s^{s+h}(r-s)\dd r+ \|\Phi\|_\infty\int_s^{s+h} \int_{-1}^0\|f(\sigma)-u^{(s)}_r(\sigma)\|\dd\sigma \dd r\\
&\quad= \tfrac{L'}{2}\|f\|_1 h^2+\|\Phi\|_\infty\int_s^{s+h} \int_{-1}^{s-r}\|f(\sigma)-f(\sigma+r-s)\|\dd\sigma \dd r\\
&\qquad\qquad\qquad+\|\Phi\|_\infty\int_s^{s+h} \int_{s-r}^0\|f(\sigma)-u^{(s)}(\sigma+r)\|\dd\sigma \dd r\\
&\quad= \tfrac{L'}{2}\|f\|_1 h^2+\|\Phi\|_\infty\int_s^{s+h} \int_{-1}^{s-r}\|f(\sigma)-f(\sigma+r-s)\|\dd\sigma \dd r\\
&\qquad\qquad\qquad+\|\Phi\|_\infty \max\bigl\{\|f\|_\infty,\|u^{(s)}|_{[s,s+1]}\|_\infty\bigr\}\tfrac{h^2}2.
\intertext{Since $f\in \W^{1,1}([-1,0];X)$, we can continue the estimation:}
&\Bigl\|t\Phi(s) f-\int_s^{s+h} \Phi(r)u^{(s)}_r\dd r\Bigr\|\leq \tfrac{L'}{2}\|f\|_1h^2+\|\Phi\|_\infty\int_s^{s+h}\|f'\|_1(r-s) \dd r\\
&\qquad\qquad\qquad+\|\Phi\|_\infty\max\bigl\{\|f\|_\infty,\|u^{(s)}|_{[s,s+1]}\|_\infty\bigr\}\tfrac{h^2}{2}\\
&\quad\leq C (\|x\|+\|f\|_1+\|f'\|_1)h^2.
\end{align*}
By summing up we obtain the estimate:
$$
\bigl\|V^{(s)}(h)x+hV^{(s)}(h)\Phi(s) f-u^{(s)}(h+s)\bigr\|\leq C(\|x\|+\|f\|_1+\|f'\|_1)h^2.
$$
Hence the assertion for the first coordinate in \eqref{eq:ize} is proved.

\medskip\noindent
Let us now turn our attention to the second coordinate of \eqref{eq:ize}. For $r\in [-1,0]$ we have the following:
\begin{align*}
\intertext{If $h+r\geq 0$}
&\big(V^{(s)}_hx+hV^{(s)}_h\Phi(s) f+\LS(h)f-u^{(s)}_{s+h}\big)(r)\\
&\qquad\qquad=V^{(s)}(h+r)(x+h\Phi(s) f)+0-u^{(s)}(s+h+r),
\intertext{and if $h+r<0$}
&\big(V^{(s)}_hx+hV^{(s)}_h\Phi(s) f+\LS(h)f-u^{(s)}_{s+h}\big)(r)\\
&\qquad\qquad=0-0+f(h+r)-f(h+r)=0.
\end{align*}
We estimate the $\Ell^1$-norm of
$$
V^{(s)}_hx+hV^{(s)}_h\Phi(s) f+\LS(h)f-u^{(s)}_{s+h}.
$$
By using the pointwise estimate for the integrand proved above for the first coordinate, we obtain that
\begin{align*}
&\big\|V^{(s)}_hx+hV^{(s)}_h\Phi f+\LS(h)f-u^{(s)}_{s+h}\big\|_1\\
&=\int_{-h}^0\|V^{(s)}(h+r)x+hV^{(s)}(h+r)\Phi(s) f-u^{(s)}(s+h+r)\|\dd r\\
&\leq\int_{-h}^0\|V^{(s)}(h+r)x+(h+r)V^{(s)}(h+r)\Phi(s) f-u^{(s)}(s+h+r)\|\dd r\\
&\qquad+\int_{-h}^0\|rV^{(s)}(h+r)\Phi(s) f \|\dd r\\
&\leq C(\|x\|+\|f\|_1+\|f'\|_1)\int_{-h}^0 (h+r)^{2}\dd r+\|\Phi\|_\infty \|f\|_1\frac{h^2}2\\
&= C(\|x\|+\|f\|_1+\|f'\|_1)h^2.
\end{align*}
The proof is hence complete.
\end{proof}
We now can  prove the first order convergence of the sequential splitting. Passing from local error estimates to convergence is done by the standard trick of telescopic summation.

\begin{theorem}
For every $T_0>0$ there is  constant $C>0$ such that  for all $f\in\W^{1,1}([-1,0];X)$ and $x=f(0)$ the inequality
\begin{equation*}
\Bigl\|\prod_{j=0}^{n-1} \calT^{(s+jh)}(h)\calS^{(s+ jh)}(h)\tvect{x}{f}-\mathcal{W}(t,s)\tvect{x}{f}\Bigr\| \leq \frac{C(t-s)^2}{n}\bigl(\|x\|+\|f\|_1+\|f'\|_1\bigr) \quad
\end{equation*}
holds  for all $s\in [-T_0,T_0]$, $t\in [s,s+T_0]$ and for all $n\in \NN$, where $h=\frac{t-s}{n}$.
\end{theorem}
\begin{proof}
Take $n_0\in \NN$ so large that $T_0/n_0<1$ holds. Fix  $n\geq n_0$, $t,s$, and $h$ as in the assertion.  By telescopic summation we obtain
\begin{align}
\label{eq:telenonaut}
& \prod_{j=0}^{n-1}\calT^{(s+ jh)}_{}(h)\calS^{(s+ jh)}(h)\tvect{x}{f}-\mathcal{W}(t,s)\tvect{x}{f}\\
\notag&\qquad=\prod_{j=0}^{n-1}\calT^{(s+ jh)}_{}(h)\calS^{(s+ jh)}(h)\tvect{x}{f}-\prod_{j=0}^{n-1}\mathcal{W}(s+(j+1)h,s+jh)\tvect{x}{f}\\
\notag&\qquad=\sum_{k=0}^{n-1}\Biggl(\Bigl(\prod_{j=k+1}^{n-1} \calT^{(s+ jh)}(h)\calS^{(s+ jh)}(h)\Bigr)\times\\
\notag&\qquad\qquad\times\bigl[\calT^{(s+ kh)}(h)\calS^{(s+ kh)}(h)-\mathcal{W}(s+(k+1)h,s+kh)\bigr]\times\\
\notag&\qquad\qquad\times \Bigl(\prod_{j=0}^{k-1}\mathcal{W}(s+(j+1)h,s+jh)\Bigr)\tvect{x}{f}\Biggr)=\\
\notag&\qquad=\sum_{k=0}^{n-1}\Biggl(\Bigl(\prod_{j=k+1}^{n-1} \calT^{(s+ jh)}(h)\calS^{(s+ jh)}(h)\Bigr)\times\\
\notag&\qquad\qquad\times\bigl[\calT^{(s+ kh)}(h)\calS^{(s+ kh)}(h)-\mathcal{W}(s+(k+1)h,s+kh)\bigr]\times\\
\notag&\qquad\qquad\times \mathcal{W}(s+kh,s)\tvect{x}{f}\Biggr).
\end{align}
For $$\tvect{x_k}{f_k}:=\mathcal{W}(s+kh,s)\tvect{x}{f}$$ we have
$\|x_k\|,\|f_k\|\leq C\|\tvect{x}{f}\|$. Therefore we can conclude  from Theorem \ref{thm:approxt2} that
$$
\Bigl\|\Bigl(\calT^{(s+ kh)}(h)\calS^{(s+ kh)}(h)-\mathcal{W}(s+(k+1)h,s+kh)\Bigr)\tvect{x_k}{f_k}\Bigr\|\leq Ch^2\bigl(\|x\|+\|f\|_1+\|f'\|_1\bigr)
$$
holds for all $k=0,\dots, n-1$. From this and from \eqref{eq:telenonaut} it follows
$$
\Bigl\|\prod_{j=0}^{n-1} \calT^{(s+jh)}(h)\calS^{(s+ jh)}(h)\tvect{x}{f}-\mathcal{W}(t,s)\tvect{x}{f}\Bigr\|\leq nCh^2\bigl(\|x_k\|+\|f_k\|_1+\|f_k'\|_1\bigr),
$$
hence the assertion.
\end{proof}


\section{Numerical examples}

In this section we present some numerical experiments obtained by the scheme described in Theorem \ref{thm:general} and in \eqref{eq:splitting_procedure_a}. The program code is a modification of the code described in Csom\'{o}s and Nickel \cite{Csomos-Nickel}. In order to check the convergence of the numerical scheme in the nonautonomous case, and compare the solutions in the autonomous and nonautonomous cases, we will investigate the following examples.

\begin{example}
Let $X=\RR$, $B=b\in\RR$ and consider
\begin{equation*}
\left\{
\begin{aligned}
\tfrac{\dd}{\dd t} u(t)&=bu(t)+\displaystyle\int_{-1}^{0} \mu(t,\sigma)u(t+\sigma)\dd\sigma, \qquad t\ge 0, \\
u(0)&=x\in\RR, \\
u_0&=f\in\Ell^1([-1,0];\RR),
\end{aligned}
\right.
\end{equation*}
for some $\mu\in\Ell^\infty(\RR\times [-1,0])$. In this case the delay operator $\Phi(t)$ is defined by
\begin{equation*}
\Phi(t) g=\int_{-1}^{0} \mu(t,\sigma)g(\sigma)\dd\sigma
\end{equation*}
for all $g\in\Ell^1([-1,0];\RR)$. Let us choose the initial values as $x=1$ and $f(\sigma)=1-\sigma$ for $\sigma\in[-1,0]$, and $b=-1$. As a particular example we choose the following functions $\mu$:
\begin{enumerate}[a)]
\item $\mu(t,\sigma)=1$ in the autonomous case, and
\item $\mu(t,\sigma)=1-\sin t$ in the nonautonomous case,
\end{enumerate}
for $t\ge 0$ and $\sigma\in[-1,0]$.
\label{exm:1}
\end{example}

Though point delays do not fit into the framework of the theoretical part of this paper, it is worth to investigate how this splitting method performs for them.
\begin{example}
Let us consider $X=\RR$, $B=b\in\RR$ and
\begin{equation}
\left\{
\begin{aligned}
\tfrac{\dd}{\dd t} u(t)&=bu(t)+\mu(t)u(t-1), \qquad t\ge 0, \\
u(0)&=x\in\RR, \\
u_0&=f\in\Ell^1([-1,0];\RR).
\end{aligned}
\right.
\nonumber
\end{equation}
The delay operator in this case is
\begin{equation}
\Phi(t) g=\mu(t)g(-1)
\nonumber
\end{equation}
for all $g\in\W^{1,1}([-1,0];\RR)$. Let us choose the initial values again as $x=1$ and $f(\sigma)=1-\sigma$ for $\sigma\in[-1,0)$, and $b=-1$. As above, we consider again the functions  $\mu$:
\begin{enumerate}[a)]
\item $\mu(t,\sigma)=1$ in the autonomous case, and
\item $\mu(t,\sigma)=1-\sin t$ in the nonautonomous case,
\end{enumerate}
for $t\ge 0$ and $\sigma\in[-1,0]$.
\label{exm:2}
\end{example}


\subsubsection*{Convergence of the numerical scheme} In order to examine the convergence of the numerical scheme described in Theorem \ref{thm:general}, we plot the values of the split solution $u_n$ obtained by the formula
\begin{equation*}
u_n=\pi_1 \prod_{p=0}^{n-1} \begin{pmatrix} V^{(s+ph)}(h) & 0 \\ V^{(s+ph)}_{h} & \LS(h) \end{pmatrix} \begin{pmatrix} I & h\Phi(s+ph) \\ 0 & \widetilde I \end{pmatrix} \tvect{x}{f}
\end{equation*}
for different values of time steps $h$. The results for the nonautonomous case are shown on Figure \ref{fig:conv}. One can see that the split solutions converges to the exact solution if $h$ decreases. The convergence of this numerical scheme in the autonomous case has been already investigated  in Csom\'os, Nickel \cite[Section 5.3]{Csomos-Nickel}.
\begin{figure}[H]
\begin{center}
\begin{tabular}{cc}
\resizebox{6cm}{!}{\PSforPDF{\includegraphics{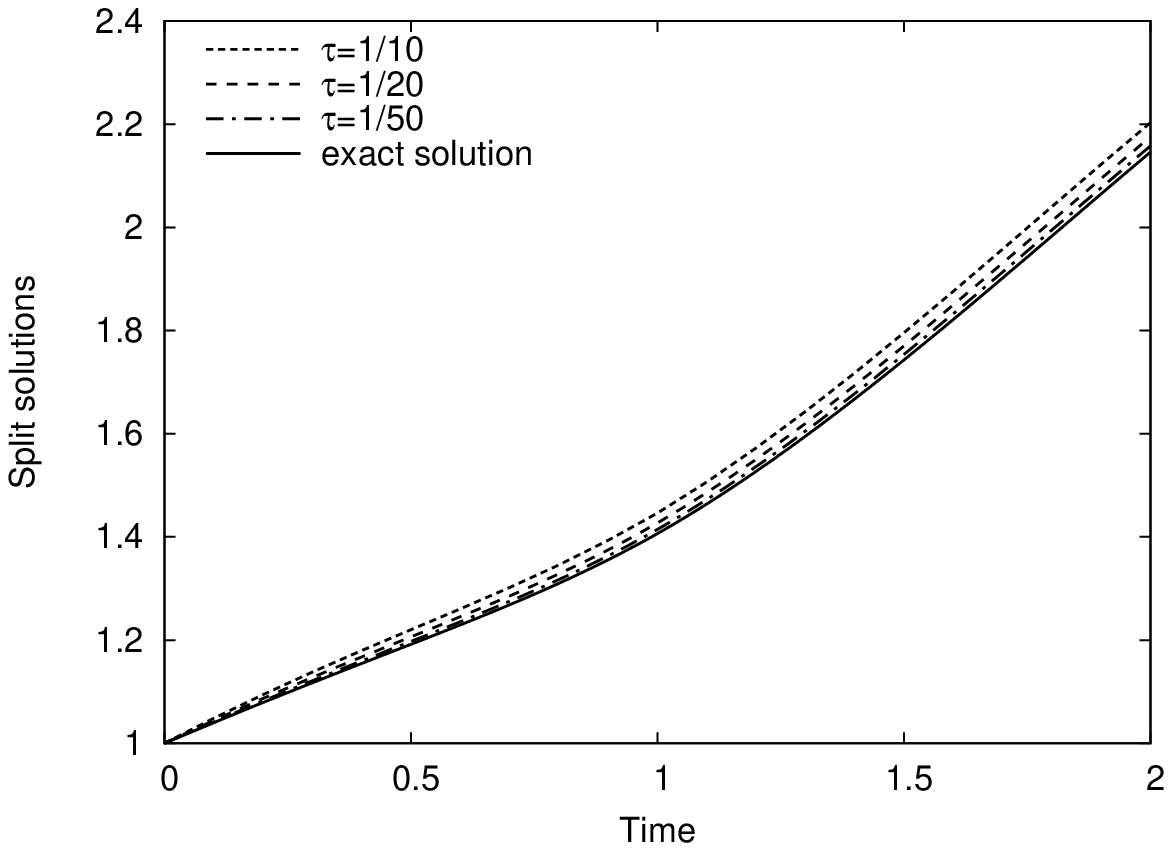}}} & \resizebox{6cm}{!}{\PSforPDF{\includegraphics{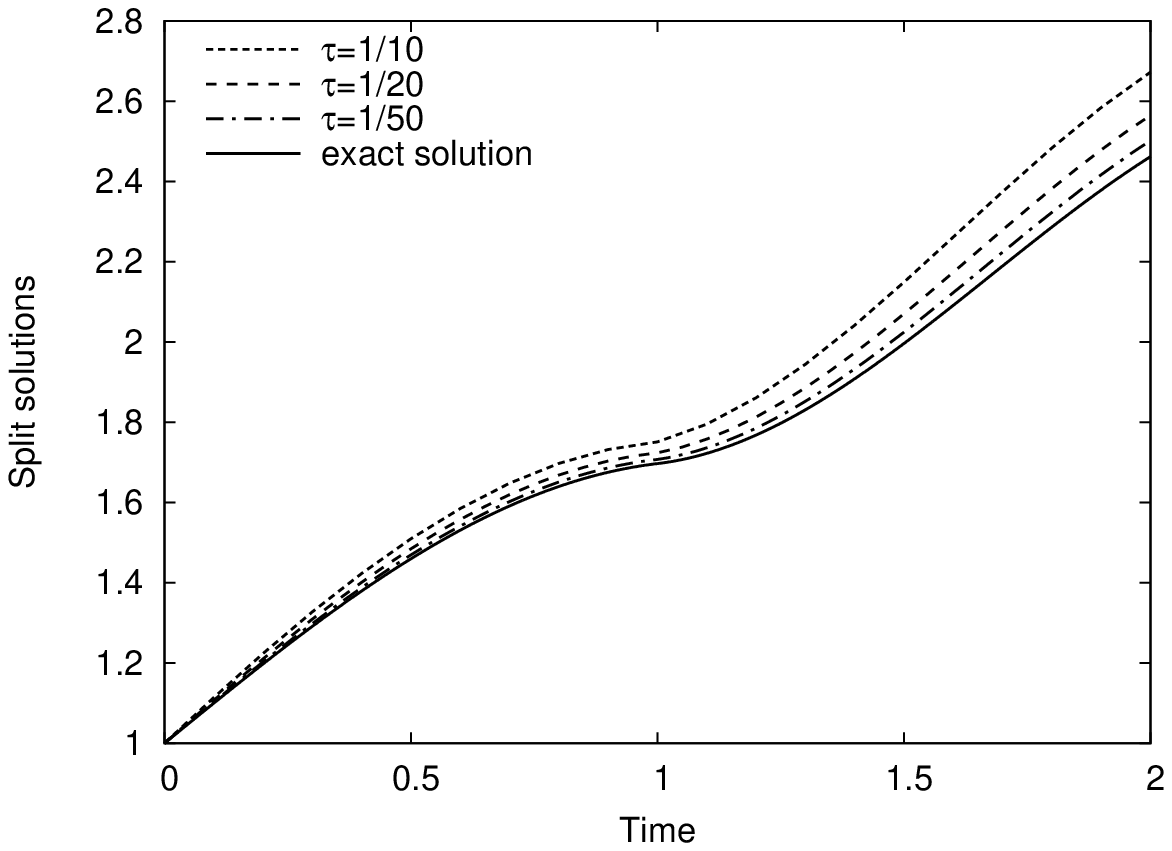}}}
\end{tabular}
\caption{\label{fig:conv} \footnotesize Results on the convergence of the numerical scheme for the nonautonomous delay equation with delay functions as in Example \ref{exm:1} (left panel) and Example \ref{exm:2} (right panel).}
\end{center}
\end{figure}

\subsubsection*{Long-time behaviour (difference between autonomous and nonautonomous cases)}
The long-time behaviour of split solutions $u_n$ of the autonomous and non\-auto\-no\-mous delay equations is shown on Figure \ref{fig:longterm} for the delay functions in Examples \ref{exm:1} (left panel) and \ref{exm:2} (right panel). It can be clearly seen that in the case of the nonautonomous equation the difference in the delay functions does not play any qualitative role, because the effect of the function $\mu$ (i.e. the sin wave) suppresses it.
\begin{figure}[H]
\begin{center}
\begin{tabular}{cc}
\resizebox{6cm}{!}{\PSforPDF{\includegraphics{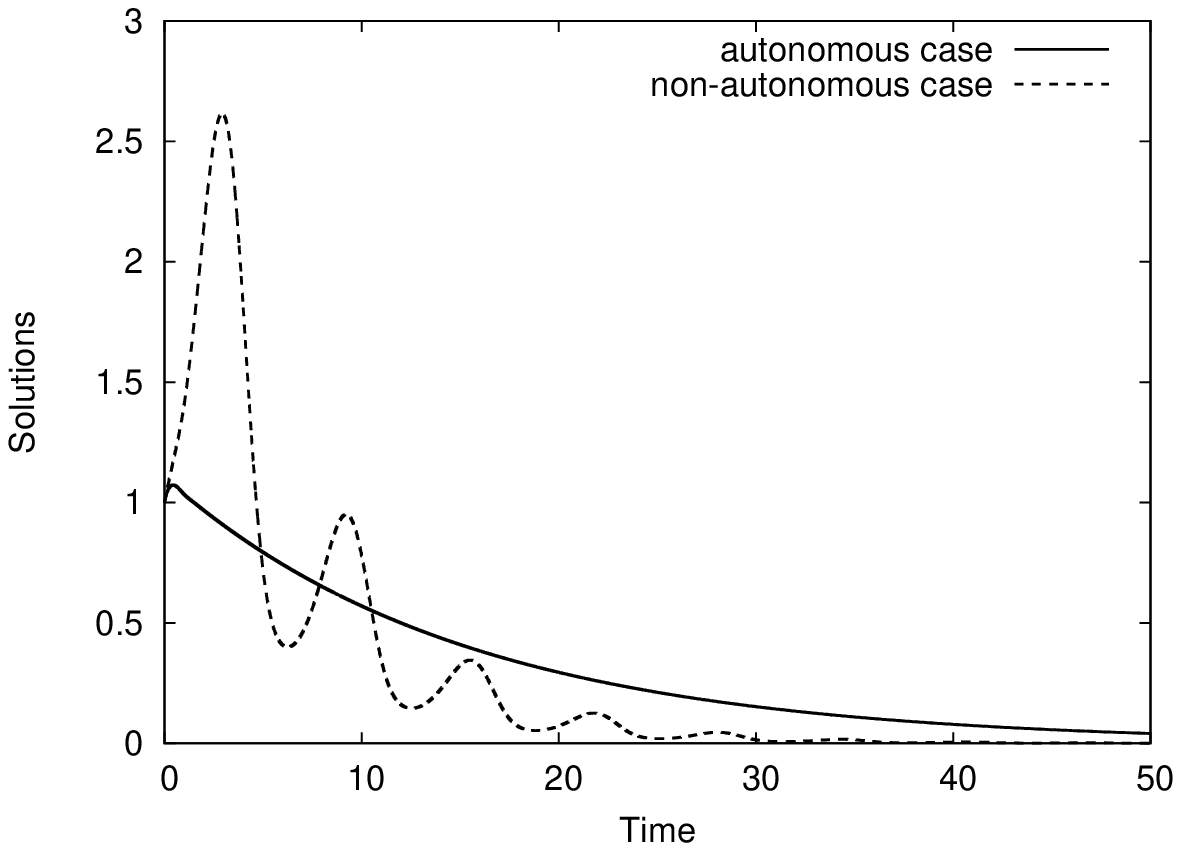}}} & \resizebox{6cm}{!}{\PSforPDF{\includegraphics{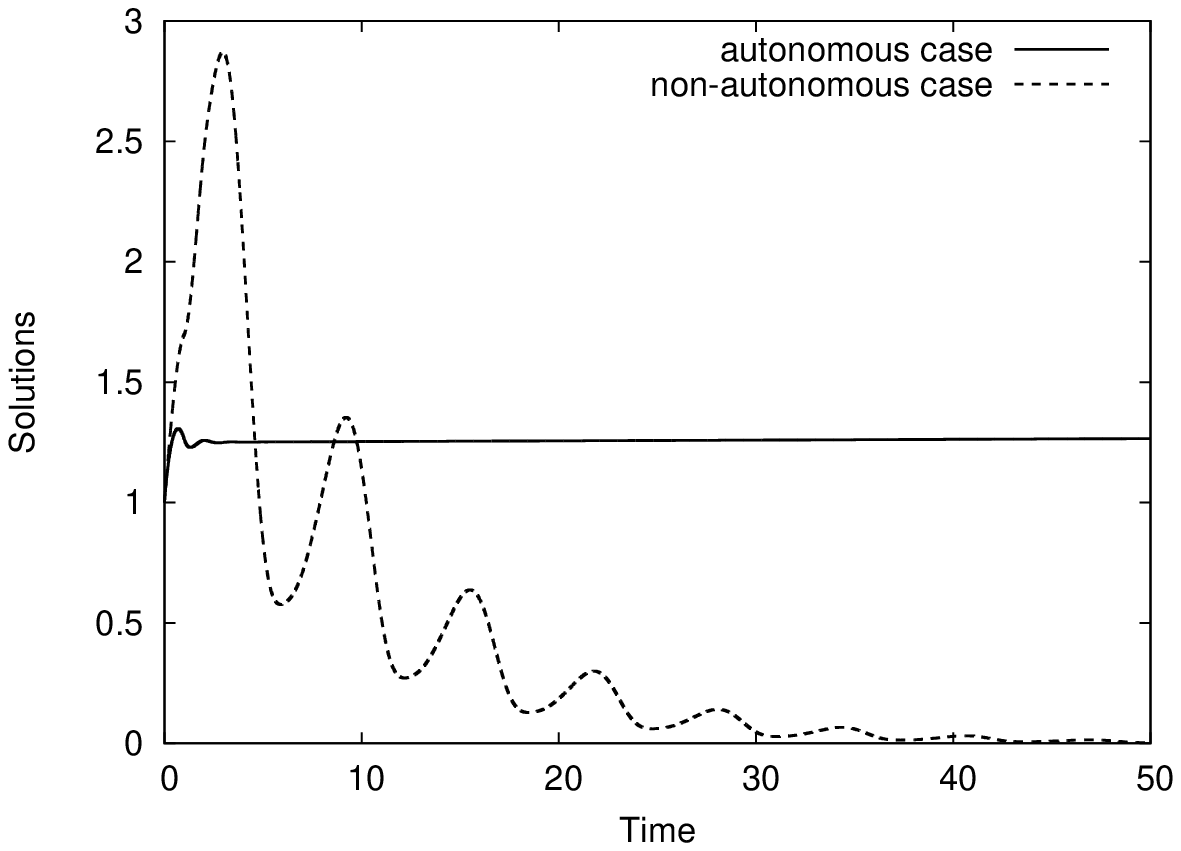}}}
\end{tabular}
\caption{\label{fig:longterm} \footnotesize Long-time behaviour of split solutions of the autonomous and nonautonomous delay equations with delay functions shown in Example \ref{exm:1} (left panel) and Example \ref{exm:2} (right panel).}
\end{center}
\end{figure}

\subsubsection*{Difference between delay functions in Examples \ref{exm:1} and \ref{exm:2}}
On Figure \ref{fig:delayfc} the effect of the different delay functions are shown in the autonomous and nonautonomous cases, respectively. As we have already seen, the effect of the delay function is suppressed by the sin wave of the function $\mu$ in the nonautonomous case. In the autonomous case, however, the (structure of the) delay function $\Phi(t)$ plays an important role.
\begin{figure}[H]
\begin{center}
\begin{tabular}{cc}
\resizebox{6cm}{!}{\PSforPDF{\includegraphics{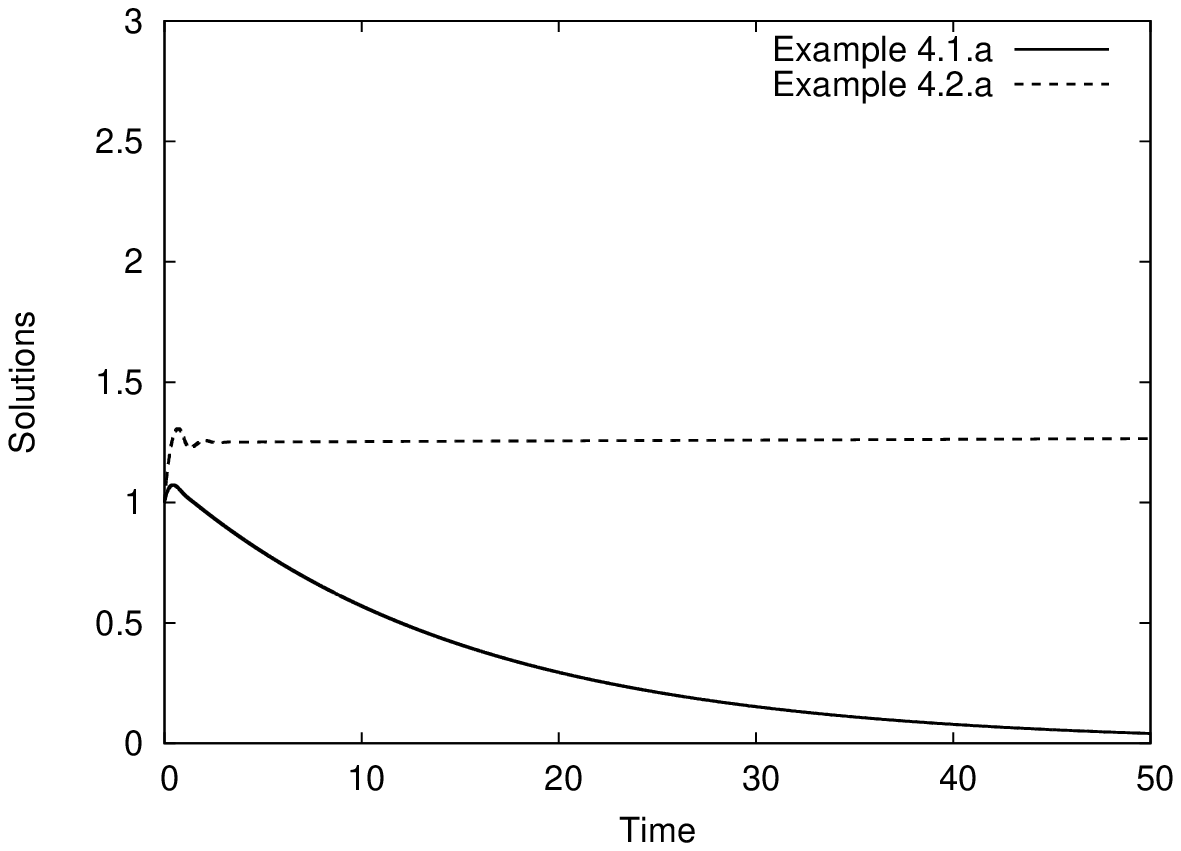}}} & \resizebox{6cm}{!}{\PSforPDF{\includegraphics{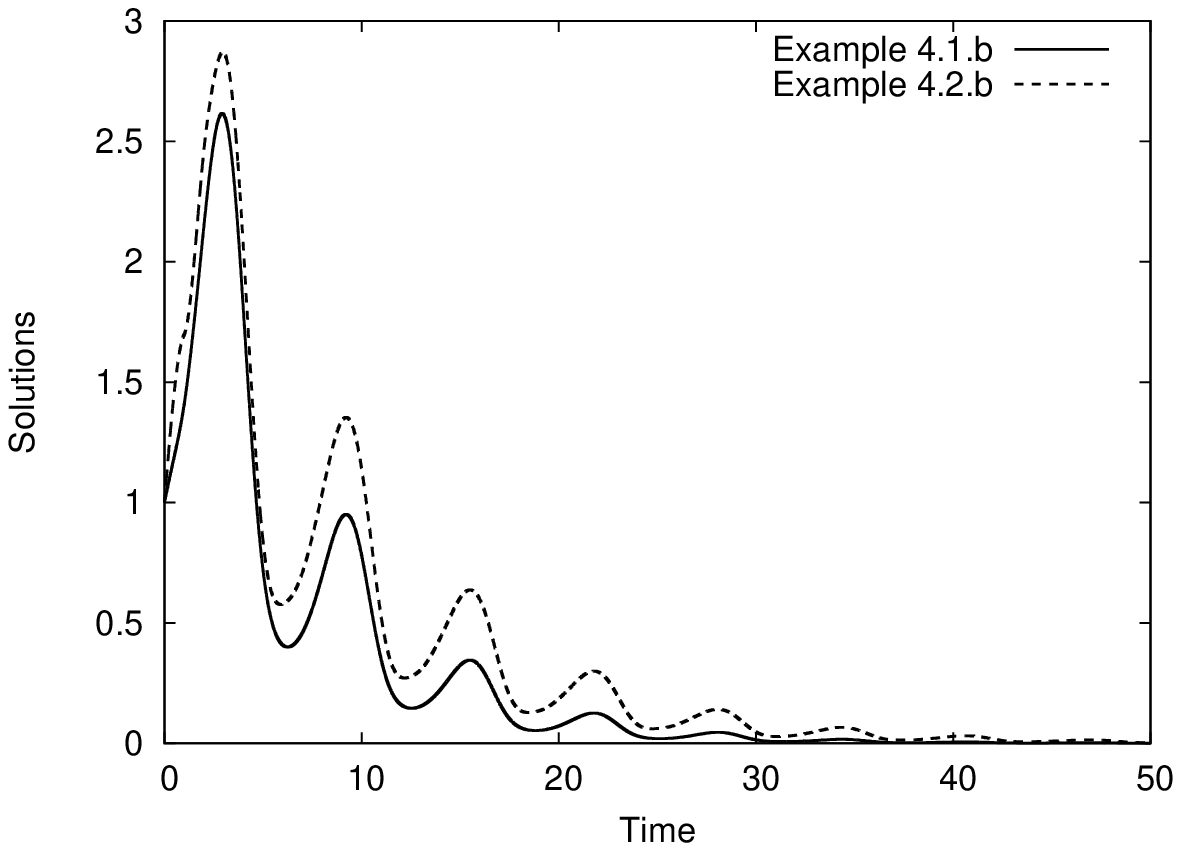}}}
\end{tabular}
\caption{\label{fig:delayfc} \footnotesize Effect of the different delay functions in the autonomous (left panel) and nonautonomous (right panel) cases.}
\end{center}
\end{figure}


\section*{Acknowledgments}
A.~B.~was supported by the Alexander von Humboldt-Stiftung and by the OTKA grant Nr. K81403. The European Union and the European Social Fund have provided financial support to the project under the grant agreement no. T\'{A}MOP-4.2.1/B-09/1/KMR-2010-0003. During the preparation of the paper B.~F.~was supported by the J\'{a}nos Bolyai Research Scholarship of the Hungarian Academy of Sciences. The financial support of the ``Stiftung Aktion \"Osterreich-Ungarn'' is gratefully acknowledged.

\parindent0pt
%

\begin{thebibliography}{10}

\bibitem{Batkai-Csomos-Farkas-Nickel1}
A.~B\'atkai, P.~Csom\'{o}s, B.~Farkas, and G.~Nickel, \emph{Operator splitting
  for non-autonomous evolution equations}, J. Funct. Anal. \textbf{260} (2011), no.~7,
  2163--2190.

\bibitem{Batkai-Csomos-Nickel}
A.~B\'{a}tkai, P.~Csom\'{o}s, and G.~Nickel, \emph{Operator splittings and
  spatial approximations for evolution equations}, J. Evol. Equ. \textbf{9}
  (2009), no.~3, 613--636.

\bibitem{Batkai-Piazzera}
A.~B\'{a}tkai and S.~Piazzera, \emph{Semigroups for delay equations}, Research
  Notes in Mathematics, vol.~10, A K Peters Ltd., Wellesley, MA, 2005.

\bibitem{Bellen-Zennaro}
A.~Bellen and M.~Zennaro, \emph{Numerical methods for delay differential
  equations}, Numerical Mathematics and Scientific Computation, The Clarendon
  Press Oxford University Press, New York, 2003.

\bibitem{Csomos-Nickel}
P.~Csom\'{o}s and G.~Nickel, \emph{Operator splitting for delay equations},
  Comput. Math. Appl. \textbf{55} (2008), no.~10, 2234--2246.

\bibitem{Engel-Nagel}
K.-J. Engel and R.~Nagel, \emph{One-parameter semigroups for linear evolution
  equations}, Graduate Texts in Mathematics, vol. 194, Springer-Verlag, New
  York, 2000, With contributions by S. Brendle, M. Campiti, T. Hahn, G.
  Metafune, G. Nickel, D. Pallara, C. Perazzoli, A. Rhandi, S. Romanelli and R.
  Schnaubelt.

\bibitem{Farago-Havasi_book}
I.~Farag\'{o} and \'{A}.~Havasi, \emph{Operator splittings and their
  applications}, Mathematics Research Developments, Nova Science Publishers,
  New York, 2009.

\bibitem{Hadd-Rhandi-Schnaubelt}
S.~Hadd, A.~Rhandi, and R.~Schnaubelt, \emph{Feedback theory for time-varying
  regular linear systems with input and state delays}, IMA J. Math. Control
  Inform. \textbf{25} (2008), no.~1, 85--110.

\bibitem{Holden-Karlsen-Lie-Risebro}
H.~Holden, K.~H. Karlsen, K.-A. Lie, and N.~H. Risebro, \emph{Splitting methods
  for partial differential equations with rough solutions}, EMS Series of
  Lectures in Mathematics, European Mathematical Society (EMS), Z\"urich, 2010,
  Analysis and MATLAB programs.

\bibitem{Lubich08}
C.~Lubich, \emph{From quantum to classical molecular dynamics: reduced models
  and numerical analysis}, Zurich Lectures in Advanced Mathematics, European
  Mathematical Society (EMS), Z\"urich, 2008.

\bibitem{Nagel-Nickel}
R.~Nagel and G.~Nickel, \emph{Well-posedness for nonautonomous abstract
  {C}auchy problems}, Evolution equations, semigroups and functional analysis
  ({M}ilano, 2000), Progr. Nonlinear Differential Equations Appl., vol.~50,
  Birkh\"auser, Basel, 2002, pp.~279--293.

\bibitem{Pazy}
A.~Pazy, \emph{Semigroups of linear operators and applications to partial
  differential equations}, Applied Mathematical Sciences, vol.~44,
  Springer-Verlag, New York, 1983.

\end{thebibliography}

\end{document}